\documentclass[reqno, 11pt]{amsart}
\usepackage[margin = 1.3 in]{geometry}
\usepackage{amsthm,amsmath,amsfonts}
\usepackage{hyperref}
\usepackage{amssymb}
\usepackage{stmaryrd}
\usepackage{eucal}
\usepackage[libertine]{newtxmath}
\usepackage{tikz-cd}

\usepackage{mathrsfs}
\usepackage{todonotes}

\usepackage[nameinlink]{cleveref}

\newtheorem{theorem}{Theorem}[section]
\newtheorem{proposition}[theorem]{Proposition}
\newtheorem{lemma}[theorem]{Lemma}

\newtheorem{definition}[theorem]{Definition}

\DeclareMathOperator{\PGL}{PGL}

\DeclareMathOperator{\rk}{rk}

\DeclareMathOperator{\Sym}{Sym}

\DeclareMathOperator{\codim}{codim}

\author{Ray Shang}
\title{Slope Semistability of Veronese normal bundles}

\begin{document}

\begin{abstract}
A classical fact is that normal bundles of rational normal curves are well-balanced. We generalize this by proving that all Veronese normal bundles are slope semistable. We also determine the line bundle decomposition of the restriction of degree 2 Veronese normal bundles to lines and rational normal curves. 
\end{abstract}
 
\maketitle

\tableofcontents

\section{Introduction}

The normal bundle of a smooth projective variety plays a crucial role in many problems of geometry, arithmetic, and commutative algebra. 
For example, cohomological information about the normal bundle can determine whether a variety satisfies interpolation \cite[Theorem A.7]{landesmanDelPezzo} \cite{larsonVogt}. The Giesker (semi)stability and slope (semi)stability of normal bundles has received attention since the early 1980s, but has been mostly restricted to normal bundles of curves, especially those embedded in $\mathbb{P}^3$. This includes the work of Coskun-Larson-Vogt \cite{coskunlarsonvogt}, Ein-Lazarsfeld \cite{einlazarsfeld}, Eisenbud-Van de Ven \cite{eisenbudven}, Atanasov-Larson-Yang \cite{atansovLarsonYang}, Ballico-Ellia \cite{ballicoEllia}, Ellingsrud-Laksov \cite{ellingsrudlaksov}, Ghione-Sacchiero \cite{ghioneSacchiero}, Ran \cite{ran}, and Newstead \cite{newstead}. Notably, Coskun-Larson-Vogt \cite{coskunlarsonvogt} recently showed that, with few exceptions, Brill-Noether general curves in $\mathbb{P}^3$ have stable normal bundle, and Ein-Lazarsfeld \cite{einlazarsfeld} showed that an elliptic curve of degree $n+1$ in $\mathbb{P}^n$ has semistable normal bundle. To the author's knowledge, fewer work has been done to investigate Gieseker (semi)stability and slope (semi)stability for normal bundles of higher dimensional varieties (the two notions coincide for curves). Although, one example is the work of Kleppe-Miro-Roig \cite[Theorem 5.3, Theorem 5.7]{KleppeMiróRoig}, which shows that when $X \subset \mathbb{P}^n$ is a smooth linear determinantal scheme with certain constraints on its codimension, its normal bundle is slope (semi)stable. 

\subsection{Main Results}

In this paper, we expand the literature on varieties with slope semistable normal bundles towards higher dimensional varieties. Let $k$ be an algebraically closed field of characteristic 0 and let $V$ be a $k$-vector space of dimension $n+1$. Recall that a degree $d$ Veronese variety of dimension $n$ is the image of an embedding 
\[
v_{n,d}: \mathbb{P}(V) \hookrightarrow \mathbb{P}(\Sym^d V)
\]
given by the complete linear series $|\mathcal{O}_{\mathbb{P}(V)}(d)|$. A rational normal curve $R$ of degree $d$ is a degree $d$ Veronese variety of dimension 1, and a classically known fact is that its normal bundle $\mathcal{N}_{R/\mathbb{P}^d}$ is isomorphic to 
\[
\mathcal{O}_{\mathbb{P}^1}(d+2)^{\oplus (d-1)}. 
\]
Thus, the normal bundle of a rational normal curve is well-balanced. We generalize this classically known fact to the following: 

\begin{theorem}\label{mainResult1}
    All Veronese normal bundles are slope semistable. 
\end{theorem} 

Since Theorem \ref{mainResult1} holds, the Grauert-Mulich theorem imposes restrictions on the line bundle decomposition of the restriction of Veronese normal bundles to a general line \cite[Theorem 3.0.1]{HuybrechtsLehn}. In the case of degree 2 Veroneses, we determine this line bundle decomposition exactly. 

\begin{theorem}\label{restrictionTheoremLine}
Let $k$ be an algebraically closed field of characteristic 0. Let $v_{n,2}: \mathbb{P}(V) \to \mathbb{P}(\Sym^2 V)$ be a degree $2$ Veronese embedding, where $V$ is a $k$-vector space of dimension $n+1$ and $n \geq 2$. Let $X$ denote the Veronese variety and $L \subset \mathbb{P}(V)$ a line. Then
\[
\mathcal{N}_{X/\mathbb{P}(\Sym^2 V)}|_L \cong \mathcal{O}_{\mathbb{P}^1}(2)^{\oplus [\frac{n(n-1)}{2}] } \oplus \mathcal{O}_{\mathbb{P}^1}(3)^{\oplus (n-1)} \oplus \mathcal{O}_{\mathbb{P}^1}(4).
\]
\end{theorem} 

We also determine the line bundle decomposition of Veronese normal bundles restricted to rational normal curves.

\begin{theorem}\label{restrictionTheoremRNC}
Let $k$ be an algebraically closed field of characteristic 0. Let $v_{n,2}: \mathbb{P}(V) \to \mathbb{P}(\Sym^2 V)$ be a degree $2$ Veronese embedding, where $V$ is a $k$-vector space of dimension $n+1$. Let $X$ denote the Veronese variety and let $R$ denote a rational normal curve of degree $n$ in $\mathbb{P}(V)$. Then 
    \[
    \mathcal{N}_{X/\mathbb{P}(\Sym^2 V)}|_R  \cong \bigoplus_{i=1}^{\frac{n(n+1)}{2}} \mathcal{O}_{\mathbb{P}^1}(2n+2). 
    \]
\end{theorem}

Facts of the type of Theorems \ref{restrictionTheoremLine} and \ref{restrictionTheoremRNC} can be applied to understand the restrictions of normal bundles to curves of higher degree and genus. For example, Proposition \ref{restrictionTheoremRNC} was used in \cite{shang} to show the existence of certain auxiliary curves which helped proved interpolation for degree 2 Veronese varieties of odd dimension.

\subsection{Outline and Notation} This paper is structured as follows. In Section \ref{preliminaries}, to keep our paper relatively self-contained, we recall the basic theory of Gieseker and slope (semi)stability and reproduce some examples of Gieseker semistable sheaves. These examples will aid us in proving Theorem \ref{mainResult1} in Section \ref{mainSection}. We prove Theorem \ref{restrictionTheoremLine} and Theorem \ref{restrictionTheoremRNC} in Section \ref{restrictionSection}.
\par
Let us establish notation used throughout this paper. In the literature, the terms $\mu$-(semi)stable and slope (semi)stable are used interchangeably -- in this paper, we only use the term slope (semi)stable. Furthermore, sometimes the terms (semi)stable and Gieseker (semi)stable are used interchangeably in the literature -- in this paper, we only use the term Gieseker (semi)stable. Throughout this paper we work over an algebraically closed field $k$ of characteristic 0. For convenience, unless specified otherwise, we fix the embedding of a dimension $n$ and degree $d$ Veronese variety to be given as follows: let V be a $k$ vector space of dimension $n+1$. Let $v_{n,d}: \mathbb{P}(V) \hookrightarrow \mathbb{P}(\Sym^d V)$ be the $d$-th power embedding so that $v_{n,d}([L]) = [L^d]$, where $L \subset V$ is a one dimensional subspace. Let $X$ denote the Veronese image under $v_{n,d}$, where the $n$ and $d$ associated to $X$ will always be clear from context.

\section{Preliminaries}\label{preliminaries}

In this section, we recall the basic theory of Giesker and slope (semi)stability and, to keep our paper relatively self-contained, we sketch a known example that helps us prove Theorem \ref{mainResult1}. This entire section follows Sections 1.1 to 1.4 of \cite{HuybrechtsLehn}. 

\subsection{Gieseker and Slope Semistability, Harder-Narasimhan filtration}

By the Birkhoff-Grothendieck theorem \cite{birkhoffGrothendieck}, we understand vector bundles over $\mathbb{P}^1$ quite well. Any vector bundle $\mathcal{V}$ over $\mathbb{P}^1$ decomposes into a direct sum of line bundles, and thus admits a decomposition
\[
\mathcal{V} \cong \bigoplus_{i=1}^{n_1} \mathcal{O}_{\mathbb{P}^1}(a_1) \oplus \cdots \oplus \bigoplus_{i=1}^{n_t} \mathcal{O}_{\mathbb{P}^1}(a_t),
\]
where $a_1 >  \cdots > a_t$. There is a natural filtration here, where the first filtration block is $\bigoplus_{i=1}^{n_1} \mathcal{O}_{\mathbb{P}^1}(a_1)$, the second filtration block is $\bigoplus_{i=1}^{n_1} \mathcal{O}_{\mathbb{P}^1}(a_1) \oplus \bigoplus_{i=1}^{n_2} \mathcal{O}_{\mathbb{P}^1}(a_2)$, and so on. Any endomorphism of the vector bundle $\mathcal{V}$ must respect this filtration.  
\par
In general, the Birkhoff-Grothendieck theorem does not hold for smooth projective $k$-varieties beyond $\mathbb{P}^1$. For example, the tangent bundle $T \mathbb{P}^n$ for projective space does not split for $n > 1$. However, the filtration we have described here generalizes to the Harder-Narasimhan filtration. Before we define the Harder-Narasimhan filtration, let us first define what a Gieseker (semi)stable sheaf is. First, we require the notion of purity of a sheaf.

\begin{definition}
Let $X$ be a Noetherian scheme. We say $E \in Coh(X)$ is pure of dimension $d$ if the support of $E$ has dimension $d$ and, for every nonzero subsheaf $F \subset E$, the support of $F$ has dimension $d$.
\end{definition}

Let $X$ be a projective $k$-scheme and let $E$ be a coherent sheaf on $X$ of dimension $d$. Fixing a very ample line bundle $\mathcal{O}_X(1)$, we can write the Hilbert polynomial of $E$ as
\[
P(E)(m) := \chi(E(m)) = \sum_{i=0}^d \alpha_i(E) \frac{m^i}{i!}
\]
where $\alpha_i(E)$ are integers and $\alpha_d(E)$ is positive. This follows from Lemma 1.2.1 of \cite[Page 9]{HuybrechtsLehn}. We define the reduced Hilbert polynomial to be $\rho(E) := \frac{P(E)}{\alpha_d(E)}$. 

\begin{definition}
Let $X$ be a projective $k$-scheme and $E \in Coh(X)$. Then $E$ is Gieseker semistable if it is pure and for every nonzero subsheaf $F \subset E$, 
\[
\rho(F) \leq \rho(E). 
\]
Replacing the inequality $\leq$ by a strict inequality $<$ gives the definition of a Gieseker stable sheaf. 
\end{definition}

A closely related notion is slope (semi)stability.

\begin{definition}
    Let $X$ be a projective $k$-scheme and let $E \in Coh(X)$ be of dimension $d = \dim (X)$. The rank of $E$ is defined to be 
\[
\rk(E) := \frac{\alpha_d(E)}{\alpha_d(\mathcal{O}_X)},
\]
the degree of $E$ is defined to be
\[
\deg E := \alpha_{d-1}(E) - \rk(E) \alpha_{d-1}(\mathcal{O}_X), 
\]
and the slope of $E$ is defined to be 
\[
\mu(E) := \frac{\deg (E)}{\rk(E)}.
\]
Then $E$ is slope semistable if and only if for all $F \subset E$ such that $0 < \rk(F) < \rk(E)$, we have $\rk(E) \deg (F) \leq \rk(F) \deg (E)$. Replacing the inequality $\leq$ by a strict inequality $<$ gives the definition of a slope stable sheaf.
\end{definition}

Gieseker and slope (semi)stability are related in the following way.

\begin{proposition}\cite[Lemma 1.2.13, page 14]{HuybrechtsLehn}\label{giesekerSlopeRelationship}
    When $E \in Coh(X)$ is a pure sheaf of dimension $d = \dim(X)$, then 
    \[
    \text{E is slope stable} \implies \text{E is Gieseker stable } \implies \text{E is Gieseker semistable } \implies \text{E is slope semistable}.
    \]
\end{proposition}

On a smooth projective variety, the Hirzebruch-Riemann-Roch formula implies that $\deg E = c_1(E)$, and if $X$ is reduced and irreducible and $E$ is the sheaf of sections of a rank $r$ vector bundle, then $\rk(E) = r$. Thus, for the vector bundles we work with in this paper, their slope is given by the formula $\mu(E) = \frac{c_1(E)}{\rk(E)}$. Furthermore, slope semistability is preserved under the following operations.

\begin{lemma}\cite[Chapter 3.2]{HuybrechtsLehn}\label{slopeSemistablePreservedTensor}
    The tensor product of slope semistable sheaves is again slope semistable. 
\end{lemma}

Furthermore, on a sufficiently nice scheme, the dual of a slope semistable vector bundle is again slope semistable. This fact is difficult to locate in the literature, so we provide a complete proof here. 

\begin{lemma}\label{slopeSemistableDualPreserved}
    Let $X$ be a normal integral projective $k$-variety, and let $E$ be a slope semistable vector bundle on $X$. Then the dual $E^*$ is also slope semistable.
\end{lemma}
\begin{proof}
    Let $F \subseteq E^*$ be a subsheaf with $0 < \rk(F) < \rk E^*$. We would like to show that 
    \[
    \rk(E^*) \deg (F) \leq \rk(F) \deg (E^*).
    \]
    By expanding and canceling like terms, we see that showing this inequality is equivalent to showing the inequality
    \[
    \alpha_d(E^*) \alpha_{d-1}(F) \leq \alpha_d(F) \alpha_{d-1}(E^*),
    \]
    where $d = \dim X$. Let $F'$ denote the saturation of $F$ with respect to $E^*$, where the saturation is defined in \cite[Definition 1.1.5]{HuybrechtsLehn}. In particular, the cokernel $F'/F$ is supported on a codimension 1 locus of $X$. Thus, $\alpha_d(F) = \alpha_d(F')$ and $\alpha_{d-1}(F') \geq \alpha_{d-1}(F)$. Now consider the short exact sequence 
    \[
    0 \to F' \to E^* \to E^*/F' \to 0.
    \]
   Since $E^*/F'$ is pure of dimension $d$, it is torsion free. Thus, since $X$ is a normal integral variety over an algebraically closed field, by \cite[Proposition 5.1.7]{shihoko}, $E^*/F'$ is locally free on an open $U \subset X$ such that $\codim (X \setminus U) \geq 2$. Restricting to $U$ yields the short exact sequence
    \[
    0 \to F'|_U \to E^*|_U \to (E^*/F')|_U \to 0
    \]
    of locally free sheaves, and dualizing yields the short exact sequence
    \[
    0 \to (E^*/F')|_U^* \to E|_U \to (F'|_U)^* \to 0.
    \]
    Note that since $\codim (X \setminus U) \geq 2$, the Hilbert polynomial coefficients $\alpha_d$ and $\alpha_{d-1}$ of $(E^*/F'|_U)^*$ and $E|_U$ are the same as those of $(E^*/F')^*$ and $E$, respectively. Thus,  
    \[
    \rho((E^*/F'|_U)^*) = \rho(E^*/F')^* \leq \rho(E) = \rho(E|_U).
    \]
    Since $E|_U$ and $(E^*/F')|_U^*$ are locally free, we can immediately deduce that $\rho(E^*/F'|_U) \geq \rho(E^*|_U)$, which is equivalent to
    \[
    \alpha_d(E^*/F'|_U) \alpha_{d-1}(E^*|_U) \leq \alpha_d(E^*|_U) \alpha_{d-1}(E^*/F'|_U).
    \]
    Using again that $\codim(X \setminus U) \geq 2$, this implies  
    \[
    \alpha_d(E^*/F') \alpha_{d-1}(E^*) \leq \alpha_d(E^*) \alpha_{d-1}(E^*/F').
    \]
    From this, since $\alpha_d(F') + \alpha_d(E^*/F') = \alpha_d(E^*)$ and $\alpha_{d-1}(F') + \alpha_{d-1}(E^*/F') = \alpha_{d-1}(E^*)$, we obtain 
    \[
    \alpha_d(E^*) \alpha_{d-1}(F') \leq \alpha_d(F') \alpha_{d-1}(E^*),
    \]
    which implies the desired inequality.
\end{proof}
We are now ready to define the Harder-Narasimhan filtration. First, every pure coherent sheaf admits a maximal destabilizing subsheaf.

\begin{lemma} \cite[Lemma 1.3.5, Page 16]{HuybrechtsLehn}
Let $E$ be a pure coherent sheaf of dimension $d$. Then there exists a subsheaf $F \subset E$ such that for all subsheaves $G \subset E$, one  has $\rho(F) \geq \rho(G)$, and in case of equality, $F \subset G$. This $F$ is called the maximal destabilizing subsheaf, and is uniquely determined and Gieseker semistable. 
\end{lemma}

The existence of the maximal destabilizing subsheaf boils down to Zorn's lemma and the additivity of Hilbert polynomials with respect to short exact sequences of coherent sheaves. Now given a pure coherent sheaf $E$ and maximal destabilizing subsheaf $F$, if $E/F \neq 0$, then $E/F$ will also admit a maximal destabilizing subsheaf that can be lifted to a subsheaf of $E$ which contains $F$. Iterating this procedure yields the Harder-Narasimhan filtration. 

\begin{theorem} \cite[Theorem 1.3.4, Page 16]{HuybrechtsLehn}
Let $X$ be a projective $k$-scheme and let $E \in Coh(X)$ be pure of dimension $d = \dim(X)$. Then $E$ admits a Harder-Narasimhan filtration 
\[
0 \subset HN_1(E) \subset \cdots \subset HN_\ell(E) = E,
\]
where $HN_i(E)/HN_{i-1}(E)$ is Gieseker semistable of dimension $d$ for $1 \leq i \leq \ell$, and defining $\rho_i(E) := \rho (\frac{HN_i(E)}{HN_{i-1}(E)})$, we have 
\[
 \rho_1 > \cdots > \rho_\ell.
\]
The Harder-Narasimhan filtration exists uniquely. 
\end{theorem}

\subsection{A useful example}\label{usefulExampleSection}

In this subsection, we sketch the main ideas of a useful example from Section 1.4 of \cite[Page 19-21]{HuybrechtsLehn}, namely that $\Sym^d V \otimes \mathcal{O}_{\mathbb{P}(V)}$ on $\mathbb{P}(V)$, along with certain subbundles which we denote as $K^i_d$ for $1 \leq i \leq d+1$, are Gieseker semistable. We have included this known example to make our paper relatively self-contained as we use the Gieseker semistability of $K^{d-1}_d$ to prove Theorem \ref{mainResult1}; the reader can find the full details of this subsection through the citations we provide.
\par
The main idea behind the Gieseker semistability of the vector bundles $K^i_d$ is as follows. Suppose a group $G$ acts on a projective $k$-scheme $X$. We say a sheaf $E$ is $G$-compatible if for every $\sigma \in G$, the action of $\sigma$ on $X$ lifts to an action $\tilde{\sigma}$ on $E$ such that the diagram
\[\begin{tikzcd}
	E & E \\
	X & X
	\arrow["{\tilde{\sigma}}", from=1-1, to=1-2]
	\arrow[from=1-2, to=2-2]
	\arrow[from=1-1, to=2-1]
	\arrow["\sigma"', from=2-1, to=2-2]
\end{tikzcd}\]
commutes. If $E \in Coh(X)$ is pure, then by uniqueness of the Harder-Narasimhan filtration on $E$, each term $HN_i(E)$ of the filtration must be a $G$-invariant subsheaf of $E$. This implies that if one can list all $G$-invariant subsheaves of $E$, then the list must include the maximal destabilizing subsheaf.
\par
We now define the subbundles $K^i_d$ of $\Sym^d V \otimes \mathcal{O}_{\mathbb{P}(V)}$. Consider the Euler exact sequence 
\[
0 \to \mathcal{O}_{\mathbb{P}(V)} \to V \otimes \mathcal{O}_{\mathbb{P}(V)}(1) \to T_{\mathbb{P}(V)} \to 0
\]
on $\mathbb{P}(V)$. Tensoring by $\mathcal{O}_{\mathbb{P}(V)}(-1)$ and then dualizing yields the short exact sequence 
\[
0 \to \Omega_{\mathbb{P}(V)}(1) \xrightarrow[]{\alpha} V \otimes \mathcal{O}_{\mathbb{P}(V)} \to \mathcal{O}_{\mathbb{P}(V)}(1) \to 0. 
\]
Symmetrizing this short exact sequence to the $i$-th degree with respect to $\alpha$ yields
\[
0 \to \Sym^i \Omega_{\mathbb{P}(V)}(1) \to \Sym^iV  \otimes \mathcal{O}_{\mathbb{P}(V)} \to \Sym^{i-1}V  \otimes \mathcal{O}_{\mathbb{P}(V)}(1) \to 0, 
\]
and tensoring by $\mathcal{O}_{\mathbb{P}(V)}(d-i)$ yields 
\[
0 \to \Sym^i \Omega_{\mathbb{P}(V)}(1) \otimes \mathcal{O}_{\mathbb{P}(V)}(d-i) \to \Sym^i V \otimes \mathcal{O}_{\mathbb{P}(V)}(d-i) \xrightarrow[]{\xi_i} \Sym^{i-1}V \otimes \mathcal{O}_{\mathbb{P}(V)}(d-i+1) \to 0.
\]
Then define the map $\delta^i_d: \Sym^d V \otimes \mathcal{O}_{\mathbb{P}(V)} \to \Sym^{d-i}V \otimes \mathcal{O}_{\mathbb{P}(V)}(i)$ to be the composition
\[
\Sym^d V \otimes \mathcal{O}_{\mathbb{P}(V)} \xrightarrow[]{\xi_d} \cdots \xrightarrow[]{\xi_{d-i+1}} \Sym^{d-i}V \otimes \mathcal{O}_{\mathbb{P}(V)}(i),
\]
and define $K^i_d$ to be the kernel of $\delta^i_d$. Note $K^{d+1}_d = \Sym^d V \otimes \mathcal{O}_{\mathbb{P}(V)}$. 
There is a natural $\PGL(V)$ action on $\mathbb{P}(V)$ and the sheaves $\Sym^t V \otimes \mathcal{O}_{\mathbb{P}(V)}, \mathcal{O}_{\mathbb{P}(V)}(t),$ and $\Omega_{\mathbb{P}(V)}$, so that the maps $\delta^i_d$ are $\PGL(V)$-equivariant and thus the $K^i_d$ are $\PGL(V)$-invariant subsheaves of $\Sym^d V \otimes \mathcal{O}_{\mathbb{P}(V)}$. In fact, the following is true.

\begin{proposition}\cite[Lemma 1.4.4]{HuybrechtsLehn} \label{onlyPGLVinvariantSubsheaves}
    The $K^i_d$ are the only nonzero $\PGL(V)$-invariant subsheaves of $\Sym^d V \otimes \mathcal{O}_{\mathbb{P}(V)}$, where $1 \leq i \leq d+1$.
\end{proposition}

The reason why Proposition \ref{onlyPGLVinvariantSubsheaves} holds is because of the following lemma. 

\begin{lemma}\cite[Lemma 1.4.3]{HuybrechtsLehn}\label{k1dNoInvariantSubsheaves}
$K^1_d = \Sym^d \Omega_{\mathbb{P}(V)}(1)$ has no proper $\PGL(V)$-invariant subsheaves. 
\end{lemma}

The idea behind Lemma \ref{k1dNoInvariantSubsheaves} is that the isotropy subgroup of $\PGL(V)$ that fixes a point $x \in \mathbb{P}(V)$ acts on the fiber of $\Sym^d \Omega_{\mathbb{P}(V)}(1)$ over $x$. A subgroup of this isotropy subgroup induces an irreducible representation of this fiber of $\Sym^d \Omega_{\mathbb{P}(V)}(1)$. Combining this irreducibility with the transitivity of the $\PGL(V)$ action on $\mathbb{P}(V)$ implies that any proper $\PGL(V)$-invariant subsheaf of $K^1_d$ is actually trivial. 
\par
Now note there exists a short exact sequence 
\[
0 \to K^i_d \to K^{i+1}_d \to K^{1}_{d-i}(i) \to 0,
\]
for $0 < i < i+1 \leq d+1$, since $\delta^{i+1}_d = \delta^{1}_{d-i}(i) \circ \delta^i_d$. Thus, successive terms $K^i_d$ and $K^{i+1}_d$ are related by $K^1_{d-i}(i)$. Combining this relation between successive terms with Lemma \ref{k1dNoInvariantSubsheaves} implies Proposition \ref{onlyPGLVinvariantSubsheaves}.
\par
By \cite[Lemma 1.4.2]{HuybrechtsLehn}, the slopes of the $K^i_d$ are such that
\[
\mu(K^1_d) < \cdots < \mu(K^d_d) < \mu(K^{d+1}_d) = 0,
\]
and combining this slope calculation with Lemma \ref{onlyPGLVinvariantSubsheaves} yields the following. 

\begin{lemma}\cite[Lemma 1.4.5]{HuybrechtsLehn}
\label{usefulExample}
Each $K^i_d$ is Gieseker semistable, where $1 \leq i \leq d+1$.  
\end{lemma}

\section{Proving Veronese normal bundles are slope semistable}\label{mainSection}

In this section, we prove our main result:

\begin{theorem}\label{theorem:veroneseSemistability}
Let $v_{n,d}: \mathbb{P}(V) \to \mathbb{P}(\Sym^d V)$ be a degree $d$ Veronese embedding, where $V$ is a $k$-vector space of dimension $n+1$. Let $X$ denote the Veronese variety. The normal bundle $\mathcal{N}_{X/\mathbb{P}(\Sym^d V)}$ is slope semistable.
\end{theorem}

Since all Veronese varieties of dimension $n$ and degree $d$ are projectively equivalent, it suffices to prove Theorem \ref{theorem:veroneseSemistability} with respect to a particular embedding. Recall that we fix our embedding to be the $d$-th power embedding, so that $v_{n,d}([L]) = [L^d]$ for one-dimensional subspaces $L \subset V$. This choice of embedding simplifies our proof. To begin, we first show that $\mathcal{N}_{X/\mathbb{P}(\Sym^d V)} \otimes \mathcal{O}_{\mathbb{P}(V)}(-d)$ is the quotient of certain locally free sheaves.

\begin{lemma}\label{normalTwistSES}
The locally free sheaf $\mathcal{N}_{X/\mathbb{P}(\Sym^d V)} \otimes \mathcal{O}_{\mathbb{P}(V)}(-d)$ is expressed as a quotient by the short exact sequence
\[
0 \to V \otimes \mathcal{O}_{\mathbb{P}(V)}(1-d) \xrightarrow[]{\Theta}  \Sym^d V \otimes \mathcal{O}_{\mathbb{P}(V)} \to \mathcal{N}_{X/\mathbb{P}(\Sym^d V)} \otimes \mathcal{O}_{\mathbb{P}(V)}(-d) \to 0.
\]
With respect to the $d$-th power embedding, the map $\Theta$ can be described in the following way: let $\{ X_i \}_{i=0}^n$ be a vector space basis of $V$ and let $\{ Z_i \}_{i=0}^n$ denote the coordinate system on $V$ which is dual to $\{ X_i \}_{i=0}^n$, so that $Z_i(X_j) = \delta_{ij}$. Similarly, let $\{ Y_{i_1 \cdots i_d } \}_{0 \leq i_1 \leq \cdots \leq i_d \leq n}$ denote the coordinate system on $\Sym^d V$ dual to the basis $\{ X_{i_1} \otimes \cdots \otimes X_{i_d} \}_{0 \leq i_1 \leq \cdots \leq i_d \leq n}$. Making the natural identifications $V \cong k \langle \frac{\partial}{\partial X_0}, \cdots, \frac{\partial}{\partial X_n} \rangle$ and $\Sym^d V \cong k \langle \{ \frac{\partial}{\partial X_{i_1}} \otimes \cdots \otimes \frac{\partial}{\partial X_{i_d}} \} \rangle$, the map $\Theta$ is given by tensoring by $(\sum_{i=0}^n Z_i \frac{\partial}{\partial X_i})^{\otimes (d-1)}$.

\end{lemma}
\begin{proof}
Consider the Euler exact sequences 
\begin{equation}\label{Euler1}
0 \to \mathcal{O}_{\mathbb{P}(V)} \xrightarrow[]{\alpha}  V \otimes \mathcal{O}_{\mathbb{P}(V)}(1) \to T_{\mathbb{P}(V)} \to 0 \text{ } \text{ } \text{      and     }  
\end{equation}
\begin{equation} \label{Euler2}
0 \to \mathcal{O}_{\mathbb{P}(\Sym^d V)} \xrightarrow[]{\beta} \Sym^d V \otimes \mathcal{O}_{\mathbb{P}(\Sym^d V)}(1) \to T_{\mathbb{P}(\Sym^d V)} \to 0. 
\end{equation}
Pulling back short exact sequence \ref{Euler2} along the $d$-th power embedding $v_{n,d}$ yields the short exact sequence 
\begin{equation} \label{Euler2PulledBack}
0 \to \mathcal{O}_{\mathbb{P}(V)} \to \Sym^d V \otimes \mathcal{O}_{\mathbb{P}(V)}(d) \to T_{\mathbb{P}(\Sym^d V)}|_{X} \to 0.
\end{equation}
We can combine short exact sequences \ref{Euler1} and \ref{Euler2PulledBack} with the short exact sequence 
\[
0 \to T_{\mathbb{P}(V)} \to T_{\mathbb{P}(\Sym^d V)}|_{X} \to \mathcal{N}_{X/\mathbb{P}(\Sym^d V)} \to 0
\]
to obtain the short exact sequence 
\begin{equation}\label{normalSES}
    0 \to V \otimes \mathcal{O}_{\mathbb{P}(V)}(1) \xrightarrow[]{\Theta'} \Sym^d V \otimes \mathcal{O}_{\mathbb{P}(V)}(d) \to \mathcal{N}_{X/\mathbb{P}(\Sym^d V)} \to 0.
\end{equation}
To be more clear on how short exact sequence \ref{normalSES} is constructed, let us be more precise about what the map $\Theta'$ is. First, recall how the map $\alpha$ in short exact sequence $\ref{Euler1}$ is typically defined. If $\{X_i \}_{i=0}^n$ of $V$ is a basis, then we can naturally identify $V$ with $k\langle \frac{\partial}{\partial X_0}, \cdots, \frac{\partial}{\partial X_n} \rangle$. Let $\{ Z_i \}_{i=0}^n$ denote the coordinate system on $V$ which is dual to the basis $\{ X_i \}_{i=0}^n$. Then $\alpha$ is given by the global section 
\[
\sum_{i=0}^n Z_i \frac{\partial}{\partial X_i} \in H^0(\mathbb{P}(V), V \otimes \mathcal{O}_{\mathbb{P}(V)}(1)).
\]
\par
Similarly, let $\{ Y_{i_1 \cdots i_d } \}_{0 \leq i_1 \leq \cdots \leq i_d \leq n}$ denote the coordinate system on $\Sym^d V$ which is dual to the basis $\{ \frac{\partial}{\partial X_{i_1}} \otimes \cdots \otimes \frac{\partial}{\partial X_{i_d}} \}_{0 \leq i_1 \leq \cdots \leq i_d \leq n}$. Then the map $\beta$ in short exact sequence \ref{Euler2} is given by the global section 
\[
\sum_{0 \leq i_1 \leq \cdots \leq i_d \leq n} Y_{i_1 \cdots i_d} [ \frac{\partial}{\partial X_{i_1}} \otimes \cdots \otimes \frac{\partial}{\partial X_{i_d} }]\in H^0(\mathbb{P}(\Sym^d V), \Sym^d V \otimes \mathcal{O}_{\mathbb{P}(\Sym^d V)}(1)).
\]
Now if we pull the map $\beta$ back along the $d$-th power embedding $v_{n,d}$, we see that there is a commutative diagram
\[\begin{tikzcd}
	{\mathcal{O}_{\mathbb{P}(V)}} & {V \otimes \mathcal{O}_{\mathbb{P}(V)}(1)} \\
	{\mathcal{O}_{\mathbb{P}(V)}} & {\Sym^d V \otimes \mathcal{O}_{\mathbb{P}(V)}(d)}
	\arrow[from=1-1, to=1-2]
	\arrow["Id"', from=1-1, to=2-1]
	\arrow["\Theta'"', from=1-2, to=2-2]
	\arrow[from=2-1, to=2-2]
\end{tikzcd}\]
where $\Theta'$ is given by tensoring by $(\sum_{i=0}^n Z_i \frac{\partial}{\partial X_i})^{\otimes (d-1)}$. 
\par
Finally, we can twist down short exact sequence \ref{normalSES} by $\mathcal{O}_{\mathbb{P}(V)}(-d)$ to obtain the desired short exact sequence
\[
0 \to V \otimes \mathcal{O}_{\mathbb{P}(V)}(1-d) \xrightarrow[]{\Theta} \Sym^d V \otimes \mathcal{O}_{\mathbb{P}(V)} \to \mathcal{N}_{X/\mathbb{P}(\Sym^d V)} \otimes \mathcal{O}_{\mathbb{P}(V)}(-d) \to 0.
\]
In particular, the map $\Theta$ is also given by tensoring by $(\sum_{i=0}^n Z_i \frac{\partial}{\partial X_i})^{\otimes (d-1)}$. 

\end{proof}

In general, if we tensor the Euler exact sequence 
\[
0 \to \mathcal{O}_{\mathbb{P}(V)} \to V \otimes \mathcal{O}_{\mathbb{P}(V)}(1) \to  T_{\mathbb{P}(V)} \to 0
\]
on $\mathbb{P}(V)$ by $\mathcal{O}_{\mathbb{P}(V)}(-1)$ to obtain the short exact sequence 
\begin{equation}\label{EulerPV-1}
    0 \to \mathcal{O}_{\mathbb{P}(V)}(-1) \to V \otimes \mathcal{O}_{\mathbb{P}(V)} \xrightarrow[]{\eta}  T_{\mathbb{P}(V)}(-1) \to 0,
\end{equation}
we can symmetrize short exact sequence \ref{EulerPV-1} to the $i$-th degree with respect to $\eta$ to obtain the short exact sequence 
\[
0 \to \Sym^{i-1} V \otimes \mathcal{O}_{\mathbb{P}(V)}(-1) \to \Sym^i V \otimes \mathcal{O}_{\mathbb{P}(V)} \to \Sym^i[T_{\mathbb{P}(V)}(-1)] \to 0.
\]
Tensoring by $\mathcal{O}_{\mathbb{P}(V)}(i-d)$ yields
\[
0 \to \Sym^{i-1} V \otimes \mathcal{O}_{\mathbb{P}(V)}(i-d-1) \xrightarrow[]{\phi_i} \Sym^i V \otimes \mathcal{O}_{\mathbb{P}(V)}(i-d) \to \Sym^i[T_{\mathbb{P}(V)}(-1)] \otimes \mathcal{O}_{\mathbb{P}(V)}(i-d) \to 0.
\]
Then we claim the following about these maps $\phi_i$.

\begin{proposition}\label{dualCompositions}
    For $1 \leq i \leq d$, dualizing the composition of maps 
\[
\Sym^{d-i}V \otimes \mathcal{O}_{\mathbb{P}(V)}(-i) \xrightarrow[]{\phi_{d-i+1}} \cdots \xrightarrow[]{\phi_d} \Sym^d V \otimes \mathcal{O}_{\mathbb{P}(V)} 
\]
gives exactly the composition of maps 
\[
\Sym^d V \otimes \mathcal{O}_{\mathbb{P}(V)} \xrightarrow[]{\xi_d}   \cdots \xrightarrow[]{\xi_{d-i+1}} \Sym^{d-i} V \otimes \mathcal{O}_{\mathbb{P}(V)}(i)
\]
defining the map $\delta^i_d$ from Subsection \ref{usefulExampleSection}.
\end{proposition}

To prove Proposition \ref{dualCompositions}, it is enough to show that dualizing $\phi_{d-j}$ gives exactly the map 
\[
\xi_{d-j}: \Sym^{d-j}V \otimes \mathcal{O}_{\mathbb{P}(V)}(j) \to \Sym^{d-j-1} V \otimes \mathcal{O}_{\mathbb{P}(V)}(j+1)
\]
defined in Subsection \ref{usefulExampleSection}. Recall that $\xi_{d-j}$ is defined by also starting with short exact sequence \ref{EulerPV-1}:
\[
0 \to \mathcal{O}_{\mathbb{P}(V)}(-1) \to V \otimes \mathcal{O}_{\mathbb{P}(V)} \xrightarrow[]{\eta}  T_{\mathbb{P}(V)}(-1) \to 0.
\]
However, instead of immediately symmetrizing with respect to $\eta$, we dualize first to obtain 
\[
0 \to \Omega_{\mathbb{P}(V)}(1) \xrightarrow[]{\eta^*}  V \otimes \mathcal{O}_{\mathbb{P}(V)} \to \mathcal{O}_{\mathbb{P}(V)}(1) \to 0,
\]
and then symmetrize the short exact sequence to degree $(d-j)$ with respect to $\eta^*$. Finally, we tensor by $\mathcal{O}_{\mathbb{P}(V)}(j)$ and consider the quotient map of the resulting short exact sequence. Working locally, we see that it suffices to show that the following holds.

\begin{lemma}\label{dualSymmetrizingCommute}
	Let 
	\[
	0 \to M \xrightarrow[]{\phi} N \xrightarrow[]{\psi} P \to 0
	\]
	be a short exact sequence of free $A$-modules. There are two equivalent ways to obtain the short exact sequence
	\[
	0 \to \Sym^i P^* \to \Sym^i N^* \to \Sym^{i-1} N^* \otimes M^* \to 0: 
	\]
	\begin{enumerate}
		\item by first symmetrizing to the $i$-th degree with respect to $\psi$ to obtain 
		\[
		0 \to \Sym^{i-1}N \otimes M \to \Sym^i N \to \Sym^i P \to 0
		\]
		and then dualizing to obtain
		\[
		0 \to \Sym^i P^* \to \Sym^i N^* \to \Sym^{i-1} N^* \otimes M^* \to 0,
		\]
		\item or by first dualizing to obtain 
		\[
		0 \to P^* \xrightarrow[]{\psi^*} N^* \xrightarrow[]{\phi^*} M^* \to 0
		\]
		and then symmetrizing to the $i$-th degree with respect to $\psi^*$ to obtain 
		\[
		0 \to \Sym^i P^* \to \Sym^i N^* \to \Sym^{i-1}N^* \otimes M^* \to 0.
		\]
	\end{enumerate}
\end{lemma}
\begin{proof}

Let us first verify that both procedures produce the same map 
\[
\Sym^i P^* \to \Sym^i N^*. 
\]
First, let us describe this map through the first procedure. The map $\Sym^i N \to \Sym^i P$ is given on pure tensors by
\[
e_1 \otimes \cdots \otimes e_i \mapsto \psi(e_1) \otimes \cdots \otimes \psi(e_i). 
\]
Then the map $\Sym^i P^* \to \Sym^i N^*$ is given by the composition 
\[
\Sym^i P^* \cong (\Sym^i P)^* \to (\Sym^i N)^* \cong \Sym^i N^*. 
\]
Now we describe the map $\Sym^i P^* \to \Sym^i N^*$ through the second procedure. First, $\psi^*$ maps $\ell \in P^*$ to $\psi^*(\ell) \in N^*$ such that for $n \in N$, we have $\psi^*(\ell)(n) = \ell(\psi(n))$. Then $\Sym^i P^* \to \Sym^i N^*$ is the map given on pure tensors by 
\[
\ell_1 \otimes \cdots \otimes \ell_i \mapsto \psi^*(\ell_1) \otimes \cdots \otimes \psi^*(\ell_i). 
\]
Now we check that the two maps from $\Sym^i P^*$ to $ \Sym^i N^*$ we have just described are actually the same. It suffices to check agreement on pure tensors. Let $\ell_1 \otimes \cdots \otimes \ell_i \in \Sym^i P^*$ and $v_1 \otimes \cdots \otimes v_i \in \Sym^i N$. The first procedure sends the pure tensor $\ell_1 \otimes \cdots \otimes \ell_i$ to the element of $\Sym^i N^*$ which evaluates pure tensors as such:
\[
v_1 \otimes \cdots \otimes v_i \mapsto \psi(v_1) \otimes \cdots \otimes \psi(v_i) \mapsto \frac{1}{i!} \sum_{\sigma \in S_i} \ell_{\sigma(1)}(\psi(v_1)) \cdots \ell_{\sigma(i)}(\psi(v_i)),
\]
where $S_i$ is the symmetric group on $i$ letters. The second procedure maps the pure tensor $\ell_1 \otimes \cdots \otimes \ell_i$ to the element $\psi^*(\ell_1) \otimes \cdots \otimes \psi^*(\ell_i)$ of $\Sym^i N^*$, which evaluates pure tensors as such: 
\[
v_1 \otimes \cdots \otimes v_i \mapsto \frac{1}{i!} \sum_{\sigma \in S_i} \ell_{\sigma(1)}(\psi(v_1)) \cdots \ell_{\sigma(i)}(\psi(v_i)).
\]
Thus, the two procedures produce the same map from $\Sym^i P^*$ to $\Sym^i N^*$.
\par
Now we verify that the two procedures produce the same quotient map $\Sym^i N^* \to \Sym^{i-1}N^* \otimes M^*$. It suffices to verify this on pure tensors. The first procedure maps the pure tensor  $\ell_1 \otimes \cdots \otimes \ell_i \in \Sym^i N^*$ to the element of $\Sym^{i-1}N^* \otimes M^*$ which evaluates elements $[a_1 \otimes \cdots \otimes a_{i-1}] \otimes b \in \Sym^{i-1}N \otimes M$ as such:
\[
[a_1 \otimes \cdots \otimes a_{i-1}] \otimes b  \mapsto a_1 \otimes \cdots \otimes a_{i-1} \otimes \phi(b) \mapsto \frac{1}{i!} \sum_{\sigma \in S_i} \ell_{\sigma(1)}(a_1) \cdots \ell_{\sigma(i-1)}(a_{i-1}) \ell_{\sigma(i)}(\phi(b)).
\]
On the other hand, the second procedure maps the pure tensor $\ell_1 \otimes \cdots \otimes \ell_i$ to the element 
\[
\frac{1}{i} \sum_{k=1}^i [\ell_1 \otimes \cdots \otimes \widehat{\ell_k} \otimes \cdots \otimes \ell_i] \otimes \phi^* (\ell_k) \in \Sym^{i-1} N^* \otimes M^*,
\]
which evaluates elements $[a_1 \otimes \cdots \otimes a_{i-1}] \otimes b \in \Sym^{i-1} N \otimes M$ as such: 
\[
[a_1 \otimes \cdots \otimes a_{i-1}]  \otimes b \mapsto \frac{1}{i}[ \frac{\ell_1(\phi(b))}{(i-1)!}  \sum_{\sigma \in S_{i-1} } \ell_{\sigma(2)}(a_1) \cdots \ell_{\sigma(i)}(a_{i-1})  + \cdots
\]
\[
 + \frac{\ell_i(\phi(b))}{(i-1)!}  \sum_{\sigma \in S_{i-1}} \ell_{\sigma(1)}(a_1) \cdots \ell_{\sigma(i-1)}(a_{i-1}) ]
\]
\[
=  \frac{1}{i!} \sum_{\sigma \in S_i} \ell_{\sigma(1)}(a_1) \cdots \ell_{\sigma(i-1)}(a_{i-1}) \ell_{\sigma(i)}(\phi(b)).
\]
Thus, the two procedures produce the same quotient map $\Sym^i N^* \to \Sym^{i-1} N^* \otimes M^*$.
\end{proof}
Hence, Lemma \ref{dualSymmetrizingCommute} implies that the dual of $\phi_{d-j}$ is $\xi_{d-j}$, which implies Proposition \ref{dualCompositions}. Now, using the same notation used in Lemma \ref{normalTwistSES}, note that each map 
\[
\phi_{d-j}: \Sym^{d-j-1}V \otimes \mathcal{O}_{\mathbb{P}(V)}(-j-1) \to \Sym^{d-j} V \otimes \mathcal{O}_{\mathbb{P}(V)}(-j)
\]
is given by tensoring by $\sum_{i=0}^n Z_i \frac{\partial}{\partial X_i}$. This implies that the composition of maps 
\[
V \otimes \mathcal{O}_{\mathbb{P}(V)}(1-d) \xrightarrow[]{\phi_2} \cdots \xrightarrow[]{\phi_d} \Sym^d V \otimes \mathcal{O}_{\mathbb{P}(V)}
\]
is given by tensoring by $(\sum_{i=0}^n Z_i \frac{\partial}{\partial X_i})^{\otimes (d-1)}$ and thus is exactly the map $\Theta$ in Lemma \ref{normalTwistSES}. This implies that if we dualize the short exact sequence 
\[
0 \to V \otimes \mathcal{O}_{\mathbb{P}(V)}(1-d) \xrightarrow[]{\Theta} \Sym^d V \otimes \mathcal{O}_{\mathbb{P}(V)} \to \mathcal{N}_{X/\mathbb{P}(\Sym^d V)} \otimes \mathcal{O}_{\mathbb{P}(V)}(-d) \to 0
\]
from Lemma \ref{normalTwistSES}, then we obtain the short exact sequence 
\[
0 \to [\mathcal{N}_{X/\mathbb{P}(\Sym^d V)} \otimes \mathcal{O}_{\mathbb{P}(V)}(-d)]^* \to \Sym^d V \otimes \mathcal{O}_{\mathbb{P}(V)} \xrightarrow[]{\delta^{d-1}_d} V \otimes \mathcal{O}_{\mathbb{P}(V)}(d-1) \to 0, 
\]
where the quotient map is $\delta^{d-1}_d$ by Proposition \ref{dualCompositions}. This means that
\[
[\mathcal{N}_{X/\mathbb{P}(\Sym^d V)} \otimes \mathcal{O}_{\mathbb{P}(V)}(-d)]^* \cong K^{d-1}_d.
\]
Hence, $[\mathcal{N}_{X/\mathbb{P}(\Sym^d V)} \otimes \mathcal{O}_{\mathbb{P}(V)}(-d)]^*$ is Gieseker semistable by Lemma \ref{usefulExample} and thus slope semistable by Proposition \ref{giesekerSlopeRelationship}. Finally, this implies that $\mathcal{N}_{X/\mathbb{P}(\Sym^d V)}$ is slope semistable by Lemma \ref{slopeSemistablePreservedTensor} and Lemma \ref{slopeSemistableDualPreserved}.

\section{Restrictions of 2-Veronese normal bundles to lines and rational normal curves}\label{restrictionSection}

By the Grothendieck-Birkhoff theorem, we know that a Veronese normal bundle restricted to a line will be isomorphic to a direct sum decomposition of line bundles. The slope semistability of Veronese normal bundles imposes certain restrictions on this direct sum decomposition for general lines. 

\begin{theorem} \cite[Theorem 3.0.1, Page 57]{HuybrechtsLehn}
Let $E$ be a slope semistable locally free sheaf of rank $r$ on $\mathbb{P}^n$. If $L$ is a general line in $\mathbb{P}^n$ and $E|_L \cong \mathcal{O}_L(b_1) \oplus \cdots \oplus \mathcal{O}_L(b_r)$ such that $b_1 \geq b_2 \geq \cdots \geq b_r$, then 
\[
0 \leq b_i - b_{i+1} \leq 1
\]
for all $i = 1, \cdots, r-1$. 
\end{theorem}
This is known as the Grauert-Mulich theorem, in its original form. A more general statement can be found in \cite[Theorem 3.1.2, Page 59]{HuybrechtsLehn}. A standard calculation shows that, if $\xi$ is the hyperplane class in the Chow ring $CH(\mathbb{P}^n)$, then the Chern class of the normal bundle for a degree $d$ Veronese of dimension $n$ is 
\[
c(\mathcal{N}_{X/\mathbb{P}(\Sym^d V)}) = \frac{(1+d \xi)^{\binom{n+d}{d}}}{(1+\xi)^{n+1}}.
\]
In particular, the degree of the first Chern class of the Veronese normal bundle restricted to a line is $\binom{n+d}{d}d - (n+1)$. Combining this Chern class information with the Grauert-Mulich theorem tells us that for a general line $L \subset \mathbb{P}(V)$, if
\[
\mathcal{N}_{X/\mathbb{P}(\Sym^d V)}|_L \cong \mathcal{O}_L(b_1) \oplus \cdots \oplus \mathcal{O}_L(b_{\binom{n+d}{d} - n - 1})
\]
such that $b_1 \geq b_2 \geq \cdots \geq b_{\binom{n+d}{d} - n - 1}$, then $0 \leq b_i - b_{i+1} \leq 1$ for $1 \leq i \leq \binom{n+d}{d} - n-2$ and $\sum_{i=1}^{\binom{n+d}{d} - n- 1} b_i = \binom{n+d}{d}d - n - 1 $. In the case of degree 2 Veroneses, we can pin down this decomposition precisely. 

\begin{theorem}\label{VeroneseNormalRestrictLine}
Let $v_{n,2}: \mathbb{P}(V) \to \mathbb{P}(\Sym^2 V)$ be a degree 2 Veronese embedding, where $V$ is a $k$-vector space of dimension $n+1$. Let $X$ denote the Veronese variety and $L \subset \mathbb{P}(V)$ a line. Then 
\[
\mathcal{N}_{X/\mathbb{P}(\Sym^2 V)}|_L \cong \mathcal{O}_L(2)^{\oplus [\frac{n(n-1)}{2} ] } \oplus \mathcal{O}_L(3)^{\oplus (n-1)} \oplus \mathcal{O}_L(4).
\]
\end{theorem}
\begin{proof}
Symmetrizing the Euler exact sequence 
\[
0 \to \mathcal{O} \to V \otimes \mathcal{O}(1) \xrightarrow[]{q} T_{\mathbb{P}(V)} \to 0 
\]
to the second degree with respect to $q$ yields the short exact sequence 
\[
0 \to V \otimes \mathcal{O}_{\mathbb{P}(V)}(1) \xrightarrow[]{\iota} \Sym^2 V \otimes \mathcal{O}_{\mathbb{P}(V)}(2) \to \Sym^2 T_{\mathbb{P}(V)} \to 0.
\]
By Lemma \ref{normalTwistSES}, we also have the short exact sequence 
\[
0 \to V \otimes \mathcal{O}_{\mathbb{P}(V)}(1) \xrightarrow[]{\Theta'} \Sym^2 V \otimes \mathcal{O}_{\mathbb{P}(V)}(2) \to \mathcal{N}_{X/\mathbb{P}(\Sym^2 V)} \to 0.
\]
Fix our embedding $v_{n,2}$ to be the second power embedding. Using the same notation used in Lemma \ref{normalTwistSES}, note that the map $\iota$ is given by tensoring by $\sum_{i=0}^n Z_i \frac{\partial}{\partial X_i}$, and note that the map $\Theta'$ is also given by tensoring by $\sum_{i=0}^n Z_i \frac{\partial}{\partial X_i}$ by Lemma \ref{normalTwistSES}.
This implies that $\Sym^2 T_{\mathbb{P}(V)}$ and $\mathcal{N}_{X/\mathbb{P}(\Sym^2 V)}$ are isomorphic. Then 
\[
\mathcal{N}_{X/\mathbb{P}(\Sym^2 V)}|_{L} \cong \Sym^2 T_{\mathbb{P}(V)}|_{L} \cong \Sym^2(T_{\mathbb{P}(V)}|_L).
\]
Thus, it suffices to identify the decomposition of $T_{\mathbb{P}(V)}|_{L}$ into line bundles. Restricting the Euler exact sequence on $\mathbb{P}(V)$ to the line $L$ yields the short exact sequence
\[
0 \to \mathcal{O}_L \to V \otimes \mathcal{O}_L(1) \to T_{\mathbb{P}(V)}|_L \to 0.
\]
Tensoring by $\mathcal{O}_{L}(-1)$ and then dualizing yields the short exact sequence
\[
0 \to K \xrightarrow[]{\alpha} V \otimes \mathcal{O}_L \to \mathcal{O}_L(1) \to 0,
\]
where $K \cong (T_{\mathbb{P}(V)}|_L \otimes \mathcal{O}_L(-1))^*$. Furthermore, we know that $K$ is a rank $n$ vector bundle isomorphic to $\bigoplus_{i=1}^n \mathcal{O}_L(a_i)$ where $\sum a_i = -1$. However, since $\alpha$ is a nontrivial injective map, we must have $K \cong \mathcal{O}_L(-1) \oplus \bigoplus_{i=1}^{n-1} \mathcal{O}_L$. This implies that $T_{\mathbb{P}(V)}|_L \cong \mathcal{O}_L(1)^{\oplus (n-1)} \oplus \mathcal{O}_L(2)$, and thus
\[
\mathcal{N}_{X/\mathbb{P}(\Sym^2 V)}|_L \cong \mathcal{O}_L(2)^{\oplus \frac{(n-1)(n-2)}{2} + (n-1) } \oplus \mathcal{O}_L(3)^{\oplus (n-1)} \oplus \mathcal{O}_L(4). 
\]
\end{proof}

Analogously, we can determine the restriction of degree 2 Veronese normal bundles to rational normal curves in $\mathbb{P}(V)$.

\begin{theorem}
    Let $v_{n,2}: \mathbb{P}(V) \to \mathbb{P}(\Sym^2 V)$ be a degree 2 Veronese embedding, where $V$ is a $k$-vector space of dimension $n+1$. Let $X$ denote the Veronese variety and let $R$ denote a rational normal curve of degree $n$ in $\mathbb{P}(V)$. Then 
    \[
    \mathcal{N}_{X/\mathbb{P}(\Sym^2 V)}|_R \cong \bigoplus_{i=1}^{\frac{n(n+1)}{2}} \mathcal{O}_{\mathbb{P}^1}(2n+2). 
    \]
\end{theorem}
\begin{proof}
As discussed in the proof of Theorem \ref{VeroneseNormalRestrictLine}, there is an isomorphism between $\Sym^2 T_{\mathbb{P}(V)}$ and $\mathcal{N}_{X/\mathbb{P}(\Sym^2 V)}$. Since 
\[
\mathcal{N}_{X/\mathbb{P}(\Sym^2 V)}|_R \cong \Sym^2 T_{\mathbb{P}(V)}|_R \cong \Sym^2(T_{\mathbb{P}(V)}|_R),
\]
it suffices to identify the decomposition of $T_{\mathbb{P}(V)}|_R$ into line bundles. Restricting the Euler exact sequence on $\mathbb{P}(V)$ to $R$ and identifying $R$ with the projective line, we obtain the short exact sequence
\begin{equation}\label{EulerPVRestrictRThenP1}
   0 \to \mathcal{O}_{\mathbb{P}^1} \xrightarrow{f} V \otimes \mathcal{O}_{\mathbb{P}^1}(n) \to T_{\mathbb{P}(V)}|_R \to 0. 
\end{equation}
We claim that $T_{\mathbb{P}(V)}|_R \cong \bigoplus_{i=1}^n \mathcal{O}_{\mathbb{P}^1}(n+1)$. 
\par 
To see this, first note that a global section $F \in H^0(\mathbb{P}^1, \mathcal{O}_{\mathbb{P}^1}(1))$ is equivalent to a map $\mathcal{O}_{\mathbb{P}^1} \to \mathcal{O}_{\mathbb{P}^1}(1)$. Then we obtain a surjective map 
\[
H^0(\mathbb{P}^1, \mathcal{O}_{\mathbb{P}^1}(1)) \otimes \mathcal{O}_{\mathbb{P}^1} \xrightarrow[]{\beta} \mathcal{O}_{\mathbb{P}^1}(1)
\]
which leads to the short exact sequence
\[
0 \to \mathcal{O}_{\mathbb{P}^1}(-1) \to H^0(\mathbb{P}^1, \mathcal{O}_{\mathbb{P}^1}(1)) \otimes \mathcal{O}_{\mathbb{P}^1} \xrightarrow[]{\beta} \mathcal{O}_{\mathbb{P}^1}(1) \to 0.
\]
Symmetrizing this short exact sequence to the $n$-th degree with respect to $\beta$ yields the short exact sequence
\[
0 \to \Sym^{n-1} H^0(\mathbb{P}^1, \mathcal{O}_{\mathbb{P}^1}(1)) \otimes \mathcal{O}_{\mathbb{P}^1}(-1) \to \Sym^n H^0(\mathbb{P}^1, \mathcal{O}_{\mathbb{P}^1}(1)) \otimes \mathcal{O}_{\mathbb{P}^1} \xrightarrow[]{\beta'} \mathcal{O}_{\mathbb{P}^1}(n) \to 0.
\]
Then tensoring by $\mathcal{O}_{\mathbb{P}^1}(-n)$ and dualizing yields 
\begin{equation}\label{evaluationSESonRNC}
    0 \to \mathcal{O}_{\mathbb{P}^1} \xrightarrow[]{g} \Sym^n H^0(\mathbb{P}^1, \mathcal{O}_{\mathbb{P}^1}(1)) \otimes \mathcal{O}_{\mathbb{P}^1}(n) \to \Sym^{n-1} H^0(\mathbb{P}^1, \mathcal{O}_{\mathbb{P}^1}(1)) \otimes \mathcal{O}_{\mathbb{P}^1}(n+1) \to 0.
\end{equation}
We claim that short exact sequence \ref{evaluationSESonRNC} is isomorphic to short exact sequence \ref{EulerPVRestrictRThenP1}. First, since rational normal curves of degree $n$ in $\mathbb{P}(V)$ are all projectively equivalent, we can fix the embedding of $R$. Let $W = k\langle Y_0, Y_1 \rangle$ be a two-dimensional vector space. Let $Z_0$ and $Z_1$ be coordinates on $W$ dual to $Y_0$ and $Y_1$, respectively. Let $\{ X_i \}_{i=0}^n$ be a basis of $V$ and choose the coordinate system on $V$ to be dual to this basis. Then fix the rational normal curve embedding $\mathbb{P}(W) \to \mathbb{P}(V)$ to be the map described on homogeneous coordinates as 
\[
[Z_0:Z_1] \mapsto [Z_0^n: Z_0^{n-1}Z_1 : \cdots : Z_1^n].
\] 
Then the map $f: \mathcal{O}_{\mathbb{P}^1} \to V \otimes \mathcal{O}_{\mathbb{P}^1}(n)$ is given by 
the global section 
\[
Z_0^n \frac{\partial}{\partial X_0} + Z_0^{n-1}Z_1  \frac{\partial}{\partial X_1}+ \cdots + Z_1^n \frac{\partial}{\partial X_n}. 
\]
\par
On the other hand, identifying $H^0(\mathbb{P}^1, \mathcal{O}_{\mathbb{P}^1}(1))$ with $k \langle Z_0, Z_1 \rangle$, the map 
\[
\beta: H^0(\mathbb{P}^1, \mathcal{O}_{\mathbb{P}^1}(1)) \otimes \mathcal{O}_{\mathbb{P}^1} \to \mathcal{O}_{\mathbb{P}^1}(1)
\]
is given by $1 \cdot Z_0 + 0 \cdot Z_1 \mapsto Z_0$ and $0 \cdot Z_0 + 1 \cdot Z_1 \mapsto Z_1$.  Then the map 
\[
\beta': \Sym^n H^0(\mathbb{P}^1, \mathcal{O}_{\mathbb{P}^1}(1)) \otimes \mathcal{O}_{\mathbb{P}^1} \to \mathcal{O}_{\mathbb{P}^1}(n)
\]
is the expected map where the induced basis elements of $\Sym^n H^0(\mathbb{P}^1, \mathcal{O}_{\mathbb{P}^1}(1))$ are sent to the elements $Z_0^n, Z_0^{n-1}Z_1, \cdots, \text{ and } Z_1^n$. Tensoring the map $\beta'$ by $\mathcal{O}_{\mathbb{P}^1}(-n)$ and dualizing, we see that $g$ and $f$ are exactly the same map. This implies that short exact sequences \ref{EulerPVRestrictRThenP1} and \ref{evaluationSESonRNC} are isomorphic, and thus
\[
T_{\mathbb{P}(V)}|_R \cong \Sym^{n-1} H^0(\mathbb{P}^1, \mathcal{O}_{\mathbb{P}^1}(1)) \otimes \mathcal{O}_{\mathbb{P}^1}(n+1) \cong \bigoplus_{i=1}^n \mathcal{O}_{\mathbb{P}^1}(n+1).  
\]
Hence,
\[
\mathcal{N}_{X/\mathbb{P}(\Sym^2 V)}|_R \cong \Sym^2(T_{\mathbb{P}(V)}|_R) \cong \bigoplus_{i=1}^{\frac{n(n+1)}{2}} \mathcal{O}_{\mathbb{P}^1}(2n+2). 
\]

\end{proof}

\section{Acknowledgements}

The author thanks Anand Patel for introducing them to the questions that led to this paper, and thanks Anand Patel and Joe Harris for their invaluable mentorship, instruction, and encouragement during the research process. Without them, this project would not have been possible. The author is grateful to Mihnea Popa, Izzet Coskun, and Aaron Landesman for helpful conversations and correspondences. The author was supported by Harvard's PRISE fellowship for a significant portion of their research time.

\bibliographystyle{alpha}
 \bibliography{bibliography}

\begin{thebibliography}{EVdV81}

\bibitem[ALY19]{atansovLarsonYang}
Atanas Atanasov, Eric Larson, and David Yang.
\newblock Interpolation for normal bundles of general curves.
\newblock {\em Memoirs of the American Mathematical Society}, 257, 2019.

\bibitem[BP84]{ballicoEllia}
E.~Ballico and Ellia P.
\newblock Some more examples of curves in $\mathbb{P}^3$ with stable normal bundle.
\newblock {\em Journal für die reine und angewandte Mathematik}, 350:87--93, 1984.

\bibitem[CLV22]{coskunlarsonvogt}
Izzet Coskun, Eric Larson, and Isabel Vogt.
\newblock Stability of normal bundles of space curves.
\newblock {\em Algebra and Number Theory}, 16:919--953, 2022.

\bibitem[ED80]{ellingsrudlaksov}
G.~Ellingsrud and Laksov D.
\newblock The normal bundle of elliptic space curves of degree 5.
\newblock {\em 18th Scandinavian Congress of Mathematicians Proceedings}, pages 258--287, 1980.

\bibitem[EL92]{einlazarsfeld}
Lawrence Ein and Robert Lazarsfeld.
\newblock Stability and restrictions of picard bundles, with an application to the normal bundles of elliptic curves.
\newblock {\em London Mathematical Society Lecture Note Series}, pages 149--156, 07 1992.

\bibitem[EVdV81]{eisenbudven}
D.~Eisenbud and A.~Van~de Ven.
\newblock On the normal bundle of smooth rational spaces curves.
\newblock {\em Mathematische Annalen}, 256:453--463, 1981.

\bibitem[GG80]{ghioneSacchiero}
F.~Ghione and Sacchiero G.
\newblock Normal bundles of rational curves in $\mathbb{P}^3$.
\newblock {\em Manuscripta Math}, 33:111--128, 1980.

\bibitem[Gro57]{birkhoffGrothendieck}
Alexander Grothendieck.
\newblock Sur la classification des fibrés holomorphes sur la sphère de riemann.
\newblock {\em American Journal of Mathematics}, 1957.

\bibitem[HL10]{HuybrechtsLehn}
Daniel Huybrechts and Manfred Lehn.
\newblock {\em The Geometry of Moduli Spaces of Sheaves}.
\newblock Cambridge University Press, 2010.

\bibitem[Ish18]{shihoko}
Shihoko Ishii.
\newblock {\em Introduction to Singularities}.
\newblock Springer, 2018.

\bibitem[KMR16]{KleppeMiróRoig}
Jan~O. Kleppe and Rosa~M. Miró-Roig.
\newblock On the normal sheaf of determinantal varieties.
\newblock {\em Journal für die reine und angewandte Mathematik (Crelles Journal)}, 2016(719):173--209, 2016.

\bibitem[LP16]{landesmanDelPezzo}
Aaron Landesman and Anand Patel.
\newblock Interpolation problems: Del pezzo surfaces.
\newblock {\em Annali Scuola Normale Superiore - Classe di Scienze}, 19, 2016.

\bibitem[LV23]{larsonVogt}
Eric Larson and Isabel Vogt.
\newblock Interpolation for brill-noether curves.
\newblock {\em Forum of Mathematics}, 2023.

\bibitem[New83]{newstead}
P.E. Newstead.
\newblock A space curve whose normal bundle is stable.
\newblock {\em Journal of London Mathematical Society}, 28:428--434, 1983.

\bibitem[Ran07]{ran}
Ziv Ran.
\newblock Normal bundles of rational curves in projective spaces.
\newblock {\em Asian Journal of Mathematics,}, 11(4):567--608, 2007.

\bibitem[Sha24]{shang}
Ray Shang.
\newblock Interpolation for degree 2 veroneses of odd dimension.
\newblock 2024.

\end{thebibliography}

\end{document}